\documentclass[11pt, letterpaper, oneside]{article}

\usepackage{latexsym, amsmath, amssymb, amsfonts, amscd}
\usepackage{amsthm}
\usepackage{t1enc}
\usepackage[mathscr]{eucal}
\usepackage{indentfirst}
\usepackage{graphicx, pb-diagram}
\usepackage{fancyhdr}
\usepackage{fancybox}
\usepackage{enumerate}
\usepackage{color}
\usepackage{hyperref}

\usepackage[all,poly,web,knot]{xy}

\theoremstyle{plain}
\newtheorem{thm}{Theorem}[section]

\newtheorem{prop}[thm]{Proposition}
\newtheorem{lemma}[thm]{Lemma}
\newtheorem{cor}[thm]{Corollary}

\renewcommand{\latticebody}{\drop@{ }}

\theoremstyle{definition}
\newtheorem{defi}[thm]{Definition}

\theoremstyle{remark}
\newtheorem{remark}[thm]{Remark}
\newtheorem{ep}[thm]{Example}

\newcommand{\Z}{\ensuremath{\mathbb Z}}

\newcommand{\R}{\ensuremath{\mathbb R}}

\newcommand{\g}{\ensuremath{\frak{g}}}

\newcommand{\cL}{{\mathcal L}}

\newcommand{\cD}{\mathcal{D}}      

\newcommand{\cO}{\mathcal{O}}
\newcommand{\cM}{\mathcal{M}}

\newcommand{\su}{\ensuremath{\mathfrak{su}}}

\newcommand{\un}{\underline}

\def\PB(#1,#2,#3,#4){\left\{\begin{matrix}#1&\!\!\!\stackrel{?}{\longrightarrow}&\!\!\!#2\\
\downarrow&&\!\!\!\downarrow\\
#3&\!\!\!\stackrel{?}{\longrightarrow}&\!\!\!#4\end{matrix}\right\}}

\def\pb(#1,#2,#3,#4){ \hom(#1 \to #3, #2 \to #4)}

\newcommand{\pd}[1]{\frac{\partial}{\partial #1}} 

\newcommand {\mcomment}[1]{{\marginpar{\textcolor{blue}{\large{$\star\star$}}}\scriptsize{\bf \textcolor{blue}{Marco}:}\scriptsize{\ \nbt{#1} \ }}}

\newcommand {\emptycomment}[1]{}

\newcommand{\nbt}[1]{\textcolor{blue}{#1}}

\begin{document}

\title{Distributions and quotients on degree $1$ NQ-manifolds {and Lie algebroids}}

\author{Marco Zambon\footnote{
Universidad Aut\'onoma de Madrid (Dept. de Matem\'aticas), and ICMAT(CSIC-UAM-UC3M-UCM),
Campus de Cantoblanco,
28049 - Madrid, Spain.
\texttt{marco.zambon@uam.es},\;\texttt{marco.zambon@icmat.es}},
Chenchang Zhu\footnote{Courant Research Centre ``Higher Order Structures'',
University of G\"ottingen, Germany.
\texttt{zhu@uni-math.gwdg.de}}
}

\date{}

\maketitle

\begin{abstract}
It is well-known that a Lie algebroid $A$ is equivalently described by a  degree 1 Q-manifold $\cM$.
We study distributions on $\cM$, giving a characterization in terms of $A$.
We show that involutive $Q$-invariant distributions on $\cM$ correspond bijectively to IM-foliations on $A$ (the infinitesimal version of Mackenzie's  ideal systems). We perform reduction by such distributions, and investigate how they arise from  non-strict actions of strict Lie 2-algebras on $\cM$.  \end{abstract}

\setcounter{tocdepth}{2}

\tableofcontents
 
 \newpage
\section*{Introduction}
\addcontentsline{toc}{section}{Introduction}
  
Lie algebroids play an important role in differential geometry and
mathematical physics. It is known that integrable Lie algebroids are
in bijection with source simply connected Lie groupoids
$\Gamma$. There is an aboundant literature studying geometric
structures on $\Gamma$ which are  compatible with the groupoid
multiplication, and how they   correspond to structures on the Lie
algebroid. Examples of the former are multiplicative foliations on
$\Gamma$, and examples of the latter are morphic foliations and IM-foliations, where ``IM'' stands
for ``infinitesimally multiplicative\footnote{This terminology was first introduced in  \cite[\S 3.1]{MaHenr}.}''

In this note we take a graded-geometric point of view on Lie algebroids, using their characterization as  NQ-1 manifolds.  The latter are  graded manifolds with coordinates of degrees $0$ and $1$, endowed with a self-commuting vector field $Q$ of degree $1$. Given a Lie algebroid $A$, we denote the corresponging NQ-1 manifold  by $\cM$.

In \S\ref{sec:di} we consider distributions on $\cM$ which are
involutive and $Q$-invariant. When regularity conditions are satisfied
we perform reduction by distributions in the graded geometry setting,
obtaining quotient NQ-1 manifolds. We show that, at the level of Lie algebroids, such distributions  
  correspond exactly to ideal systems on $A$, and that this reduction
  is the
reduction of $A$ by  ideal systems \cite{MK2}. 
Further we show that involutive, $Q$-invariant distributions
correspond bijectively to the IM-foliations on $A$ introduced   by
Jotz-Ortiz \cite{JotzOrtizFol}. To prepare the ground for the above results, at the beginning of  \S\ref{sec:di} we consider distributions on degree $1$ N-manifolds and  characterize  them in terms of ordinary differential geometry.
 
In \S\ref{sec:ac} we consider distributions that arise form certain actions. As an NQ-1 manifold $\cM$ has coordinates in degrees $1$ and $0$, its module of vector fields is generated in degrees $-1$ and $0$, and hence it is natural to act on $\cM$ by strict Lie-2 algebras $L$ (differential graded Lie algebras concentrated in degrees $-1$ and $0$).
We define such actions
  as $L_\infty$-morphisms from $L$ into the DGLA of vector fields on   $\cM$.  In general the ``image'' of an action fails to be an involutive distribution (however it is automatically preserved by $Q$). We give a sufficient condition for the involutivity, and  show that performing  reduction by the action one obtains a new NQ-1 manifold. 
  
 A possible explanation for the fact that the involutivity fails in
 general  is the following. In ordinary geometry, when a Lie algebra
 action on a manifold is  {(globally)} free and proper, the Lie
 groupoid of the transformation Lie algebroid is Morita equivalent to
 the quotient of the manifold by the action.  {Thus  the
   transformation Lie algebroid (an NQ-1 manifold) may be taken as a good replacement of the
   quotient when the action is no longer free and proper.}
In our setup it is shown in \cite{RajMarco} that
$L[1]\times \cM$ inherits a {degree 1} homological vector field
 (generalizing the notion of transformation Lie
algebroid). {The fact that} $L[1]\times \cM$ is an NQ-2 manifold suggests that 
 a  suitably defined quotient $\cM/L$ 
should not be an NQ-1 manifolds in general (it should be so only when certain ``regularity'' conditions are satisfied).
 \\

\noindent\textbf{Notation and conventions:}
$M$ always denotes a smooth manifold.  For any vector bundle $E$, we denote by $E[1]$ the N-manifold obtained from $E$  by declaring that the fiber-wise linear coordinates on $E$ have degree one. If $\cM$ is an N-manifold, we denote by $C(\cM)$
the graded commutative algebras of ``functions on $\cM$''. By
$\chi(\cM)$ we denote graded Lie algebra of vector fields on $\cM$ (i.e., graded derivations of $C(\cM)$).
The symbol $A$ always denotes a Lie algebroid over $M$.\\ 

\noindent\textbf{Acknowledgements:} 
We thank  Rajan Mehta, Pavol \v{S}evera and Jim Stasheff for their very helpful comments and for discussions. 

\noindent Zambon thanks the Courant Research Centre  ``Higher Order Structures" for hospitality.  Zambon was partially supported by the Centro de Matem\'atica da Universidade do Porto, financed by FCT through the programs POCTI and POSI,  by the FCT program Ciencia 2007, grants PTDC/MAT/098770/2008 and 
PTDC/MAT/099880/2008  (Portugal), and by MICINN RYC-2009-04065 (Spain).

\noindent
Zhu thanks C.R.M. Barcelona for hospitality. Zhu is supported by the German Research Foundation
(Deutsche Forschungsgemeinschaft (DFG)) through
the Institutional Strategy of the University of G\"ottingen.

\section{Background}

{To be self-contained,} we first recall some background material from \cite[\S 1.1]{ZZL1}.
More precisely:  we recall how the notion of degree 1 N-manifold is
equivalent to the notion of vector bundle (\S \ref{nnman}), and how
the notion of NQ-1 manifolds is equivalent to that of Lie algebroid
(\S \ref{enq}).

\subsection{N-manifolds}\label{nnman}

The notion of N-manifold (``N'' stands for non-negative) was introduced by \v{S}evera  in \cite{SWLett}\cite{s:funny}. Here we adopt the definition given by Mehta in  \cite[\S 2]{rajthesis}. Useful references are also  \cite[\S 2]{AlbICM}\cite{AlbFlRio}.

If $V=\oplus_{i< 0}V_i$ is a finite dimensional $\Z_{<0}$-graded vector space, recall that $V^*$ is the $\Z_{>0}$-graded vector space defined by $(V^*)_i=(V_{-i})^*$. We use $S^{\bullet}(V^*)$ to  denote the \emph{graded} symmetric algebra over $V^*$, so its homogeneous elements anti-commute if they both have odd degree.  $S^{\bullet}(V^*)$ is a graded commutative algebra  concentrated in positive degrees.
 
 Ordinary manifolds are modeled on open subsets of $\R^n$, and N-manifolds modeled on the following graded charts:
\begin{defi}\label{locm}
Let $V=\oplus_{i< 0}V_i$ be a finite dimensional $\Z_{<0}$-graded vector space.\\
 The \emph{local model for an N-manifold} consists  of a pair  as follows:
\begin{itemize}
\item $U\subset \R^n$ an open subset 
\item the sheaf (over $U$) of graded commutative algebras given by
$U'\mapsto C^{\infty}(U')\otimes S^{\bullet}(V^*)$.
\end{itemize} 
\end{defi}
\begin{defi}\label{def:grm}
An \emph{N-manifold} $\cM$ consists of a pair  as follows:
\begin{itemize}
\item a topological space $M$ (the ``body'')
\item a sheaf $\cO_M$ over $M$ of graded commutative algebras, locally isomorphic to the
above local model (the sheaf of ``functions'').
\end{itemize}
\end{defi}

We use the notation $C(\cM):=\cO_M(M)$ to denote the space of ``functions
on $\cM$''.
By $C_k(\cM)$ we denote the degree $k$ component of $C(\cM)$, for any non-negative $k$. The
{\em degree} of the graded manifold is the largest $i$ such that  $V_{-i}\neq\{0\}$.
Degree zero graded manifolds are just
 ordinary  manifolds:  $V=\{0\}$, and all functions have degree zero.

\begin{ep}\label{VBs}
Let $F=\oplus_{i< 0}F_{i} \rightarrow M$ be a   graded vector bundle.
The N-manifold associated to it has body $M$, and $\cO_M$ is given by the sheaf of sections of  $S^{\bullet}F^*$. \end{ep}

\begin{defi}
A \emph{vector field} on $\cM$ is a
graded derivation of the algebra\footnote{Strictly speaking one should define vector fields
in terms of the sheaf $\cO_M$ over $M$. However we will work only with objects defined on the whole of the body $M$, hence the above definition will suffice for our purposes. }
$C(\cM)$.
\end{defi}

Since $C(\cM)$ is a graded commutative algebra (concentrated in non-negative degrees), the space of vector fields $\chi(\cM)$,
equipped with the graded commutator $[-,-]$,  is a graded
Lie algebra (see Def. \ref{dgla}).

We will focus mainly on  degree 1 N-manifolds, which we now describe in more detail. To do so we recall first
\begin{defi}\label{CDO}
Given a vector bundle $E$ over $M$, a  \emph{covariant differential operator\footnote{Also known as derivative endomorphism, see \cite[\S 1]{MC-YKS}.} (CDO)}
is a linear map $X: \Gamma(E) \to \Gamma(E) $ such that there exists a vector field
$\underline{X} $  on $M$ (called \emph{symbol}) with
\begin{align}\label{eqCDO}
 X(f\cdot e) = \underline{X} (f) e + f \cdot X(e), \quad
\text{for} \; f \in C^\infty(M), \; e \in \Gamma(E).
\end{align}
We denote the set of CDOs on $E$ by
 $CDO(E)$.
If $X\in CDO(E)$, then
the dual  $X^*\in CDO(E^*)$ is defined by
\begin{equation}\label{dualCDO}
\langle X^* (\xi), e \rangle + \langle \xi, X(e) \rangle=\un{X}
(\langle \xi, e\rangle ), \quad \text{for all}\; e \in \Gamma(E), \;
\xi \in \Gamma(E^*).
\end{equation}
\end{defi}

Recall that if $E\rightarrow M$ is a (ordinary)  vector bundle, $E[1]$ denotes the graded vector bundle whose fiber over $x\in M$ is  $(E_x)[1]$ (a graded vector space concentrated in degree $-1$).

\begin{lemma}\label{N1}
If $E\rightarrow M$ is a  vector bundle, then
$\cM:=E[1] $ is a degree $1$
N-manifold with body $M$, and conversely all  degree $1$
N-manifolds arise this way.

The algebra of functions $C(\cM)$ is generated
by
$$C_0(\cM)=C^{\infty}(M) \text{  and  }
C_1(\cM)= \Gamma(E^*).$$

The $C(\cM)$-module of vector fields is generated by elements in degrees $-1$ and $0$. We have identifications
$$\chi_{-1}(\cM)=\Gamma(E)\text{  and  }\chi_0(\cM)=CDO(E^*)$$
 induced by the actions on functions. Further the map $\chi_0(\cM) \cong CDO(E)$  obtained dualizing CDOs is just $X_0\mapsto[X_0,\cdot]$ (using the identification   $\chi_{-1}(\cM)=\Gamma(E)$).
\end{lemma}
\begin{proof} See \cite{ZZL1}.  
\end{proof}

\begin{remark}\label{coordsNQ1}
Let us choose
 coordinates $\{x_i\}$ on an open subset $U\subset M$ and a  frame $\{e_{\alpha}\}$ of sections of $E|_U$.
Let $\{\xi^{\alpha}\}$
be the dual frame for $E^*|_U$, and assign degree $1$ to its elements. Then $\{x_i, \xi^{\alpha}\}$
form a set of coordinates for $\cM:=E[1]$ (in particular they generate $C(\cM)$ over $U$).
The coordinate expression of vector fields is as follows. $\chi_{-1}(\cM)$ consist of elements of the form $f_{\alpha}\pd{\xi^{\alpha}}$, and $\chi_{0}(\cM)$ of elements of the form $g_{i}\pd{x_i}+
f_{ \alpha \beta}\xi^{\alpha}\pd{\xi^{\beta}}.$ Here  $f_{\alpha},g_{i},f_{ \alpha \beta} \in C^{\infty}(M)$, for $i\le dim(M)$ and $\alpha, \beta \le rk(E)$, and we adopt the Einstein summation convention.
\end{remark}

\subsection{NQ-manifolds and Lie algebroids}\label{enq}
We will be interested in N-manifolds equipped with extra structure:
\begin{defi}\label{NQm}
An {\em NQ-manifold} is an N-manifold   equipped
with a \emph{homological vector field}, i.e. a degree 1 vector field $Q$ such that $[Q, Q]=0$.
\end{defi}

To shorten notation, we call a degree $n$ NQ-manifold a {\em NQ-$n$ manifold}.

Before considering $\chi(\cM)$, we recall the notion of  differential graded Lie algebra (DGLA):
\begin{defi}\label{dgla}
  A \emph{graded Lie algebra} consists of   a graded vector space $L=\oplus_{i\in \Z} L_i$
together with a  bilinear bracket $[\cdot,\cdot] \colon L \times L \rightarrow L$ such that, for all  homogeneous   $a,b,c\in L$:
\begin{itemize}
\item[--] the bracket is degree-preserving: $[L_i,L_j]\subset L_{i+j}$
 \item[--] the bracket is graded skew-symmetric:
 $[a,b]=-(-1)^{|a||b|}[b,a]$
\item[--] the adjoint action $[a,\cdot]$
is a degree $|a|$ derivation of the bracket (Jacobi identity): $[a,[b,c]]=[[a,b],c]+(-1)^{|a||b|}[b,[a,c]]$ .
\end{itemize}
  A \emph{differential graded Lie algebra} (DGLA) $(L,[\cdot,\cdot], \delta)$ is a graded Lie algebra together with a linear $\delta : L \rightarrow L$ such that
\begin{itemize}
 \item[--] $\delta$ is a degree $1$ derivation of the bracket:
$\delta(L_i)\subset L_{i+1}$ and $\delta[a,b]=[\delta a,b]+ (-1)^{|a|}[a, \delta b]$
\item[--]  $\delta^2=0$.
\end{itemize}
\end{defi}

\begin{lemma} \label{lemma:vf}
For a NQ-$n$ manifold  $\cM$, the space of vector fields  $$(\chi(\cM),  [Q, -], [-,-])$$
is a negatively bounded DGLA with lowest degree $-n$.
\end{lemma}

A well-known example of NQ-manifolds is given by Lie algebroids \cite{MK2}.

\begin{defi}
A {\em Lie algebroid} $A$ over a manifold $M$ is a vector
bundle over $M$, such that the global sections of $A$ form a Lie
algebra with  Lie bracket $[\cdot,\cdot]_A$ and Leibniz rule holds
\[ [a, fa']_A=f[a, a']_A + \rho_A(a)(f) a' , \quad a, a' \in \Gamma(A), f
\in C^\infty(M), \]where $\rho_A: A \to TM$  is
  a vector
bundle morphism called the  \emph{anchor}.
\end{defi}

The following is well known (\cite{vaintrob}, see also \cite{yvette}):
\begin{lemma}\label{nq1la}
NQ-1 manifolds  are in bijective correspondence with Lie algebroids.
\end{lemma}

We describe the  correspondence using the derived bracket
construction.
By Lemma \ref{N1} there is a bijection between vector bundles and
degree 1 N-manifolds. If $A$ is a Lie algebroid, then the homological vector field  is just the   Lie algebroid differential acting on $ \Gamma(\wedge^{\bullet}A^*)=C(A[1])$. Conversely, if
$(\cM:=A[1], Q)$ is an NQ-manifold, then
the Lie algebroid structure on $A$ can be recovered by the derived bracket
construction \cite[\S 4.3]{yvette}: using the
 identification $\chi_{-1}(\cM)= \Gamma(A)$ recalled in Lemma \ref{N1}, we define
\begin{equation}\label{eq:br-rho}
[a,a']_A=[[Q,a],a'], \;\;\;\;\quad  \rho_A(a)f=[[Q,a],f],
\end{equation}
where $a,a'\in \Gamma(A)$ and $f\in C^{\infty}(M)$.

In coordinates the correspondence is as follows.
Choose coordinates $x_{\alpha}$ on $M$ and a frame of sections $e_i$ of $A$, inducing (degree $1$) coordinates $\xi_i$ on the fibers of $A[1]$. Then
\begin{equation}\label{QA}
Q_A=\frac{1}{2} \xi^j\xi^i
c_{ij}^k(x)\pd{\xi_k}+\rho_{i}^{\alpha}(x)\xi^i\pd{x_{\alpha}}
\end{equation}
where $[e_i,e_j]_A=c_{ij}^k(x)e_k$ and the anchor of $e_i$ is $\rho_{i}^{\alpha}(x)\pd{x_{\alpha}}$.

\section{Distributions}\label{sec:di}

In this section we study distributions on degree 1 N-manifolds and NQ-manifolds. In \S\ref{sec:disN} we characterize in classical terms distributions on degree 1 N-manifolds, and carry out reduction by involutive distributions. In \S\ref{sec:disNQ-1} we consider involutive, Q-invariant distributions on an NQ-1 manifold. We show that in classical terms they correspond to IM-foliations. When regularity conditions are satisfied -- classically the correspond to the notion of ideal system -- we perform reduction obtaining quotient NQ-1 manifolds.   

\subsection{Distributions  on degree 1 N manifolds}
\label{sec:disN}

Let $E\to M$ be a vector bundle. We define distributions on $E[1]$, following \cite{BCMZ}.
\begin{defi}
A {\em distribution} $\cD$ on the  degree 1 N-manifold $\cM:=E[1]$ is
a graded $C(\cM)$-submodule of $\chi(\cM)$ such that
 for any $x\in M$ there exists a neighborhood $U\subset M$ and homogeneous generators of $\cD$ over $U$ such that their evaluations at every point of $U$ are $\R$-linearly independent.
A distribution is \emph{involutive} if it is closed under the Lie bracket of vector fields on $\cM$.

Here the evaluation on $M$ of a vector field on $\cM$ is its image in  $\chi(\cM)/C_{\ge 1}(\cM)\chi(\cM)$, the space of sections of the graded vector bundle $T\cM|_M=E[1]\oplus TM\to M$.
\end{defi}

\begin{remark}
A graded $C(\cM)$-submodule $\cD$ of $\chi(\cM)$ is a distribution iff there exist $l\le dim(M)$, $\lambda\le rk(E)$ and, for any $m\in M$, coordinates $\{x_i,\xi_{\alpha}\}$ on $\cM=E[1]$ defined in a neighborhood $U$ of $m$ (see Remark \ref{coordsNQ1}) such that, over $U$,
   $\cD$ is generated as a $C(\cM)$-module by
 \begin{align*}
\left\{\pd{\xi_{\alpha}}\right\}_{\alpha\le \lambda} \text{ and }\left\{\sum_{j\le dim(M)} g_i^j(x)\pd{x_j}+\sum_{ \beta,\gamma\le rk(E)}f_i^{ \beta\gamma }(x)\xi_{\beta}\pd{\xi_{\gamma}}\right\}_{i\le l}
\end{align*}
 and the $\{\sum_{j} g_i^j(x)\pd{x_j}\}_{i\le l}$ are $\R$-linearly independent at every point of $U$. Notice that there is no condition on the coefficients $f^i_{\beta\gamma }$.
\end{remark}

Let $\cD$ be a distribution on $\cM$.
Denote by $\cD_i$ the degree $i$ component of $\cD$, consisting of the degree $i$ vector fields on $\cM$  which lie in $\cD$ ($i\ge -1$).

Recall that for any vector bundle $E$, there is an associated Lie algebroid  $\mathfrak{D}(E)$ whose sections are exactly $CDO(E)$ (see Def. \ref{CDO}) endowed with the commutator bracket, and whose anchor $s \colon \mathfrak{D}(E) \to TM$ is given by the symbol
 \cite[\S 1]{MC-YKS}.  $\mathfrak{D}(E)$ fits in an exact sequence of Lie algebroids   \begin{equation}\label{ses}
0\to End(E)\to \mathfrak{D}(E) \overset{s}\to TM \to 0
\end{equation}
where
 $End(E)$  denotes the vector bundle endomorphisms of $E$ that cover $Id_M$. The following lemma gives a characterization of distributions on $\cM$ in classical terms.

\begin{lemma}\label{distrCDO}
 There is a bijection between
distributions $\cD$ on $\cM:=E[1]$ and the following data:
\begin{itemize}
 \item subbundles $B\to M$ of $E$,
\item subbundles $C\to M$ of $\mathfrak{D}(E)$ for which $ker(s)\cap C=\{\phi\in End(E): \phi(E)\subset B\}$.\end{itemize}
The correspondence is
\begin{align}
\cD \mapsto \begin{cases}  B \text{ such that }\Gamma(B)=\cD_{-1} \\
C \text{ such that }\Gamma(C)=\cD_0,
\end{cases}
\end{align}
where we identify $\Gamma(E)=\chi_{-1}(\cM)$ and $\Gamma(\mathfrak{D}(E))=CDO(E)\cong\chi_{0}(\cM)$ as in Lemma \ref{N1}.

$\cD$ is involutive iff $C\subset \mathfrak{D}(E)$ is a Lie subalgebroid   and the   action of sections of $C$ preserves $\Gamma(B)$.
\end{lemma}

\begin{proof} Let $\cD$ be a distribution.
 By definition, locally there exist integers $l\le dim(M)$, $\lambda\le rk(E)$ as well as homogeneous generators $\{X^{\alpha}_{-1}\}_{\alpha\le \lambda}$
and $\{X^i_{0}\}_{i\le l}$
of $\cD$ whose evaluations at points of $M$, which are given by  $\{X^{\alpha}_{-1}\}_{\alpha\le \lambda}$ and $ \{s(X^i_{0})\}_{i\le l}$, are linearly independent.    $$\cD_{-1}=C^{\infty}(M)\cdot\{X^{\alpha}_{-1}\}_{\alpha\le \lambda}$$ hence consists of sections of a subbundle $B\subset E$.
We have $$\cD_0=C^{\infty}(M)\cdot \{X^{i}_{0}\}_{i\le l}+ C_1(\cM)\cdot \{X^{\alpha}_{-1}\}_{\alpha\le \lambda}.$$
The linear independence condition on the generators implies that $ker(s|_{\cD_0}) = C_1(\cM)\cdot \{X^{\alpha}_{-1}\}_{\alpha\le \lambda}$ is the space of sections of a subbundle of $End(E)$, and that $s(\cD_0)$  is the space of sections of a subbundle of $TM$.
Hence $\cD_0$ is the space of sections of a
 subbundle $C\subset \mathfrak{D}(E)$, and
$ker(s)\cap C=\{\phi\in End(E): \phi(E)\subset B\}$.

Conversely, given a pair $(B,C)$ as in the statement, defining $\cD_{-1}:=\Gamma(B)$ and
$\cD_0:=\Gamma(C)$ we obtain a distribution.

Since a distribution $\cD$ is generated (as a $C(\cM)$-module) by $\cD_{-1}$ and $\cD_0$, the involutivity of $\cD$ is equivalent to $[\cD_0,\cD_{0}]\subset \cD_{0}$ and $[\cD_0,\cD_{-1}]\subset \cD_{-1}$. The first condition means that $C$ is a subalgebroid of $\mathfrak{D}(E)$, the second (by Lemma \ref{N1}) that the sections of $C$ preserve $\Gamma(B)$.
\end{proof}

 We rephrase the classical data associated to
an involutive distributions as in Lemma \ref{distrCDO}. Here and in the following, if $X_0\in  CDO(E)$, we denote its symbol  by  $\un{X_0}:=s(X_0)\in \chi(M)$.

\begin{lemma}\label{BF}
Given a subbundle   $B\to M$ of $E$, there is a bijection between 
\begin{itemize}
\item Lie subalgebroids $C\to M$ of $\mathfrak{D}(E)$ for which $ker(s)\cap C=\{\phi\in End(E): \phi(E)\subset B\}$, such that   action of sections of $C$ preserves $\Gamma(B)$.
\item pairs $(F,\nabla)$ where $F$ is an integrable distribution on $M$ and $\nabla$ is a flat $F$-connection on the vector bundle $E/B$.
\end{itemize} 
\end{lemma}
\begin{proof}
Let $C$ be as above. Then $F:=s(C)$ has constant rank, and is involutive since $C$ is a Lie subalgebroid. 
Define  
$$\nabla_{\un{X_0}}(e \text{ mod } B)=X_0(e) \text{ mod } B$$ on $E/B$, where  $X_0\in \Gamma(C)$. The above properties of $C$ make clear that $\nabla$ is well-defined, and further it is an $F$-connection.
 The flatness of $\nabla$
follows from
$$\nabla_{[\un{X_0},\un{X_0'}]}(e \text{ mod } B) =[X_0,X_0'](e) \text{ mod } B =[\nabla_{\un{X_0}},\nabla _{\un{X_0'}}](e \text{ mod } B)$$ where in the first equality we used that $C$ is a Lie subalgebroid of $\mathfrak{D}(E)$.

Conversely, given a pair $(F,B)$ as above,  
\begin{equation}\label{Cas}
 \{X_0\in \mathfrak{D}(E): \un{X_0}\in \Gamma(F),\; X_0(e)\in \Gamma(B) \text{ whenever } \nabla(e \text{ mod } B)=0 \}
\end{equation}
is the space of sections of a subbundle $C$ of $\mathfrak{D}(E)$, which fits in the short exact sequence of vector bundles
 \begin{equation*} 
0\to \{\phi\in End(E): \phi(E)\subset B\}\to C \overset{s}\to F \to 0.
\end{equation*}
To see this, notice that for any $Y\in \Gamma(F)$ there exists  $X_0$  lying in \eqref{Cas} with $\un{X_0}=Y$: choose locally a frame $\{e_i\}$ of $E$ consisting of a local frame for $B$ and lifts of $\nabla$-flat local sections of $E/B$, then just define $X_0(e_i)=0$ for all $i$, and obtain a covariant differential operator $X_0$   by imposing   eq. \eqref{eqCDO}. Hence $C$ is a subbundle of $\mathfrak{D}(E)$, and
one checks easily that it satisfies the required properties.
 \end{proof}

Let us introduce a piece of terminology: we say that an involutive distribution $F$ on a manifold $M$ is  \emph{simple} if there exists a smooth structure on $M/F$ for which the projection $\pi \colon M \to M/F$ is a submersion.

Consider $C(\cM)^{\cD}$, the sheaf (over the body $M$) of  $\cD$-invariant functions on $\cM$, defined assigning to $U\subset M$ the algebra   $$C(\cM)^{\cD}_U:=\{f\in C( \cM)_U:X(f)=0 \text{ for all }X\in \cD_U\}.$$
Assume that  $F$ is simple.
Then
\begin{equation}\label{sheafMF}
V\mapsto C(\cM)^{\cD}_{\pi^{-1}(V)}
\end{equation}
 is a presheaf  of graded commutative algebras over $M/F$. When it    defines an N-manifold (with body $M/F$), we say that  the quotient of $\cM$ by $\cD$ is \emph{smooth}, and denote the quotient N-manifold by $\cM/\cD$.

\begin{prop}\label{distriE} Let $E\to M$ be a vector bundle.
Let $\cD$ be an involutive distribution on $\cM=E[1]$.
Denote by $(B,F,\nabla)$ the data associated to $\cD$ as in Lemmas \ref{distrCDO} and \ref{BF}.
 
If $F$ is simple and 
$\nabla$ has no holonomy, then $\cM/\cD$ is an N-manifold, and  $$\cM/\cD\cong \tilde{E}[1]$$ as N-manifolds, where $\tilde{E}\to M/F$ is the vector bundle   obtained quotienting $E/B \to M$ by the action of the flat $F$-connection $\nabla$ by parallel transport. 
\end{prop}
\begin{proof}
 We compute $C(\cM)^{\cD}$. In degree zero we simply have
  $C_0(\cM)^{\cD}=C^{\infty}(M)^F$.
In degree
$1$ we have that the subset of $C_1(\cM)$ annihilated by  $\cD_{-1}$
is exactly $\Gamma(B^{\circ})$.
Now consider the $F$-connection $\nabla^*$ on
$B^{\circ}=(E/B)^*$ dual to $\nabla$, defined by
$$\langle \nabla^* \xi, V \rangle=-\langle  \xi, \nabla V \rangle+d\langle \xi,V \rangle$$
for sections $\xi$ of  $B^{\circ}$ and $V$ of $E/B$.
It is given by
 $\nabla^*_{\un{X_0}} \xi= X_0(\xi)$ where $X_0\in \cD_0$ and where we identify $\Gamma(E^*)=C_1(\cM)$ as in Lemma \ref{N1}.
 Hence
  we conclude that
 \begin{align*}
   C_1(\cM)^{\cD} =& \{\xi \in \Gamma(B^{\circ}): \nabla^*
   \xi =0 \} \\
    =& \{\xi \in \Gamma(B^{\circ}): \langle \xi,
    \Gamma_{\nabla}(E/B) \rangle \subset  C^{\infty}(M)^F\}
 \end{align*}
where $\Gamma_{\nabla}(E/B)$ denotes the space of $\nabla$-parallel
 sections of $E/B$.

By assumption, $F$ is simple and 
 $\nabla$ has no
 holonomy. On one hand, this assures that the quotient  of $E/B$ by the   action of $\nabla$ is a smooth  vector bundle.  On the other hand this implies that
 the technical conditions i) and ii) of  \cite[Lemma 5.12]{CZ2} are satisfied and hence eq. \eqref{sheafMF} gives a \emph{sheaf} of graded commutative algebras  generated by its elements in degrees $0$ and $1$.
Since by the above $C_0(\cM)\cong C^{\infty}(M/F)$ and  $C_1(\cM)^{\cD}\cong \Gamma(\tilde{E}^*)$, this implies that the sheaf given by eq. \eqref{sheafMF} corresponds to the N-manifold  $\tilde{E}[1]$. Hence
$\cM/\cD\cong \tilde{E}[1]$ as N-manifolds.
\end{proof}

 \subsection{Distributions on  NQ-1 manifolds}\label{sec:disNQ-1} 

In this subsection  $A$ is a Lie algebroid over $M$ and $\cM:=A[1]$
the corresponding NQ-1 manifold (see Lemma \ref{nq1la}), whose homological vector field we denote by $Q$. 

Let $A\to M$ be a Lie algebroid. We recall the definition of ideal system on $A$ \cite[Def. 4.4.2]{MK2}.

\begin{defi} \label{is}
An \emph{ideal system} for the Lie algebroid $A \to M$ consists of
\begin{itemize}
 \item a Lie subalgebroid $B\to M$ of $A$,
\item a closed, embedded wide\footnote{This means that the subgroupoid has the same base
$M$.} subgroupoid  of the pair groupoid $M\times M$ of the form
$R=\{(x,x'):\pi(x)=\pi(x')\}$ for some surjective submersion $\pi
\colon M \to N$,
\item a linear action $\Theta$ of $R$ on the vector bundle  $A/B\to M$,
\end{itemize}
such that, referring to a section $a\in \Gamma(A)$ as  $\Theta$-stable whenever $\Theta(a)\in \Gamma(B)$,
\begin{itemize}
 \item[(i)] if  $a,a' \in \Gamma(A)$ are $\Theta$-stable then $[a,a']_A$  is also $\Theta$-stable,
 \item[(ii)] if  $b\in \Gamma(B)$, and $a\in \Gamma(A)$ is $\Theta$-stable, then $[b,a]_A\in \Gamma(B)$,
 \item[(iii)] the anchor $\rho_A$ maps $B$ into $F:=ker(\pi_*)$,
 \item[(iv)] the map $A/B \to TM/F$ induced by the anchor $\rho_A$ is $R$-equivariant w.r.t. the action $\Theta$ of $R$ on $A/B$ and the canonical action of $R$ on $TM/F$.
\end{itemize}
\end{defi}

When $A$ is a Lie algebra, an ideal system is simply an ideal of $A$.

An ideal system always  induces a Lie algebroid structure on the quotient of $A/B$ by the action $\Theta$ (a vector bundle over $N$)  such that the natural projection is a Lie algebroid morphism   \cite[Thm. 4.4.3]{MK2}.\\

The next Proposition \ref{distri} shows that, under certain conditions, an involutive distribution on $A[1]$ preserved by the homological vector field $Q$
gives an ideal system on $A$. Let us first understand this in the case of a
Lie algebra $A=\g$.
\begin{ep}\label{gD}[Distributions on Lie algebras]
Let $A=\g$ be a Lie algebra. The degree $-1$ part of  a distribution $\cD$ on $\g[1]$
corresponds to a subspace $B\subset \g$. In suitable local
coordinates on $\g[1]$, we have
 \[\cD_{-1}=span_{\R}\left\{\pd{\xi_{\alpha}}
\right\}_{\alpha\le \dim B}, \quad \cD_{0}=span_\R \left\{ \xi_{\alpha'}
\pd{\xi_\alpha}
\right\}_{\alpha\le \dim B, \alpha'\le \dim \g }. \]
This can be seen directly from the definition of distribution or from Lemma \ref{distrCDO}.
  The distribution $\cD$ is automatically
involutive.

Further, if $[Q,\cdot]$ preserves $\cD$, then $B$ is an
ideal in $\g$. Indeed, for any index $\gamma\le dim(B)$ we have $[Q,\pd{\xi_\gamma}]\in \cD_0$. Hence for any $\beta \le dim(\g)$, using the identification $\g \cong \chi_{-1}(\g[1])$, we have
\[
\left[\pd{\xi_\gamma}, \pd{\xi_\beta} \right]_\g= \left[\left[Q,\pd{\xi_\gamma}\right],
\pd{\xi_\beta} \right]= \left[\sum_{ \alpha=1}^{\dim B}
\sum_{ \alpha'=1}^{\dim \g}
c_{\alpha' \alpha}  \xi_{\alpha'}
  \pd{\xi_{\alpha}}, \pd{\xi_\beta} \right] = -\sum_{ \alpha=1}^{\dim B}
  c_{\beta \alpha}
  \pd{\xi_\alpha} \in B
\]
for some constants $c_{\alpha' \alpha}$.
\end{ep}

\begin{prop}\label{distri} Let $A\to M$ be a Lie algebroid. Let $\cD$ be an involutive distribution on $\cM:=A[1]$, and
assume  that  $[Q,\cD]\subset
\cD$, where $Q$  is the homological vector field on $\cM$ as in Lemma \ref{nq1la}. 
Assume that $F$ is simple and 
$\nabla$ has no holonomy, where $(B,F,\nabla)$ is the data associated to $\cD$ as in Lemmas \ref{distrCDO} and \ref{BF}.

Then the following is an ideal system for the Lie
algebroid $A$:
\begin{itemize}
\item the  Lie subalgebroid $B$ of $A$,
\item the Lie subgroupoid $R$ of $M\times M$ associated to the submersion $\pi \colon M \to M/F$,
\item the linear action of $R$ on the vector bundle $A/B$ given by parallel translation w.r.t. $\nabla$.
\end{itemize}

\end{prop}
\begin{proof}
 $B$ is a Lie subalgebroid of $A$: if
$b,b'\in \Gamma(B)$ then
$$[b,b']_A=
[[Q,b],b']\in \cD_{-1}=\Gamma(B)$$ using
$[Q,\cD]\subset \cD$
and $[\cD_0,\cD_{-1}]\subset \cD_{-1}$.
If $F$ is simple then $R$ is a  closed, embedded wide subgroupoid  of $M\times M$.
If $\nabla$ has no holonomy then the groupoid action of $R$ on $A/B$ is well-defined.

To check that we indeed have an ideal system we need to check  (i)-(iv) in Def. \ref{is}. Recall that $F$ and $\nabla$ were defined in the proof of Lemma \ref{BF}.
\begin{itemize}
\item [(ii)] \emph{If  $b\in \Gamma(B)$, and $a\in \Gamma(A)$  satisfies
$[\cD_0,a]\in \Gamma(B)$, then $[b,a]_A\in \Gamma(B)$}

Indeed $[[Q,b],a]\in  \Gamma(B)$ because $Q$ preserves
$\cD$.

\item
[(iii)] \emph{If $b \in \Gamma(B)$  then $\rho_A(b)=\un{[Q,b]}\in
\Gamma(F)$.}

This is clear because $Q$ preserves $\cD$.

\item[(iv)] \emph{If $p,q$ are points of $M$ lying in the same fiber of $\pi$ and
$a_p\in A_p$, then $\pi_*(\rho_A(a_p))=\pi_*(\rho_A(a_q))$, where $a_q\in A_q$ is a lift of the $\nabla$-parallel translation of $(a_p  \text{ mod } B)$ from $p$ to $q$.}

It is enough to show that
if $a \in \Gamma(A)$ satisfies $[\cD_0,a]\in \Gamma(B)$
then $\rho_A(a)=\un{[Q,a]}$ descends to a vector field on $M/F$, i.e.
$[\un{\cD_0},\un{[Q,a]}]\subset \Gamma(F)$. To show this we proceed as follows. For any $X_0\in \cD_0$ consider the r.h.s. of
\begin{equation*}
[{X_0},[Q,a]]=[[X_0,Q],a]+[Q,[X_0,a]].
\end{equation*}
$[X_0,Q]\in
\cD_1=C_1(\cM)\cdot \cD_0$, so using the assumption on
$a$ we get
 $[[X_0,Q],a]\in \cD_0$. The second term on the r.h.s. also lies in
 $\cD_0$, since $[X_0,a]\in \Gamma(B)$ by assumption and $[Q,\cD]\subset \cD$.
Hence $[\un{X_0},\un{[Q,a]}]=\un{[{X_0},{[Q,a]}]}\in s(\cD_0)=\Gamma(F).$

\item[(i)] \emph{If  $a,a'\in \chi_{-1}(\cM)=\Gamma(A)$ satisfy $[\cD_0,a]\in
\Gamma(B)$, $[\cD_0,a']\in \Gamma(B)$ then
$[\cD_0,[a,a']_A]\in \Gamma(B)$.}

Let $X_0\in \cD_0$. Using repeatedly the Jacobi
identity we have
\begin{align*}
\left[X_0,[[Q,a],a']\right]=&\left[[X_0,[Q,a]],a'\right]+\left[[Q,a],[X_0,a']\right]\\=&
\left[[[X_0,Q],a],a'\right]+\left[[X_0,a],a'\right]_A+\left[a,[X_0,a']\right]_A.
\end{align*}
Since $[X_0,a],[X_0,a']\in \Gamma(B)$ by assumption,    (ii) implies
that the second and third term on the r.h.s. lie in $\Gamma(B)$. The
computation in (iv) shows that  $[[X_0,Q],a]\in \cD_0$, so by
the assumption on $a'$ the first term on the r.h.s. also lies in
$\Gamma(B)$.
\end{itemize}
\end{proof}

We show that  the quotient   of $\cM=A[1]$  by an involutive distribution $\cD$ preserved by $Q$ agrees with the
natural  quotient of
the Lie algebroid $A$ by the corresponding ideal system (\cite[Thm. 4.4.3]{MK2}).

\begin{prop}\label{quot}
Consider the set-up of Prop. \ref{distri}. Then  $\cM/\cD$ is an  NQ-1 manifold and $$\cM/\cD\cong \tilde{A}[1]$$ as NQ-1 manifolds, where $\tilde{A}\to M/F$ is the Lie algebroid   obtained
as the quotient of $A$ by the  ideal system defined in Prop. \ref{distri}.
\end{prop}

\begin{proof} 
We have $\cM/\cD\cong \tilde{A}[1]$ as N-manifolds
by Prop. \ref{distriE}.

The assumption $[Q,\cD]\subset
\cD$ implies that $Q$ preserves $C(\cM)^{\cD}$: if $f$ is $\cD$-invariant, then $Q(f)$ is also $\cD$-invariant, because for all $X\in \cD$ we have
$$X(Q(f))=\pm Q(X(f))+[X,Q](f)=0.$$ Hence, by restricting the action of $Q$ to
$C(\cM)^{\cD}\cong  C(\tilde{A}[1])$, we obtain a homological vector field $\tilde{Q}$ on $\tilde{A}[1]$.

By construction  the inclusion $C(\cM)^{\cD} \to C(\cM)$ respects the action of the homological vector fields, so
that the quotient Lie algebroid  structure on $\tilde{A}$ (obtained via the derived bracket construction using $\tilde{Q}$) has the property that the projection $A\to \tilde{A}$ is a Lie algebroid morphism. Hence it agrees with the Lie algebroid structure obtained
by the ideal system.
\end{proof}

We present an example where  $\cD$ is singular, i.e. just a graded $C(\cM)$-submodule of $\chi(\cM)$ but not a distribution,  and the quotient is
nevertheless a smooth NQ-manifold (even though not concentrated in
degrees $0$ and $1$).
\begin{ep}\label{suD}[Singular quotient]
Let $A$ be the Lie algebra
$\su(2,\R)$, so that in a suitable basis  we have $[a_1,a_2]=a_3$,
$[a_2,a_3]=a_1$,$[a_3,a_1]=a_2$. Denote by $\xi_i$ the coordinates
on $\cM:=A[1]$ dual to the basis $a_i$. The  homological vector field on $\cM$  is
$Q=\xi_2\xi_1\pd{\xi_3}+\xi_1\xi_3\pd{\xi_2}+\xi_3\xi_2\pd{\xi_1} $.

On $\cM$ consider  the $C(\cM)$-span $\cD$ of  $\frac{\partial}{\partial \xi_1}$  and $[Q, \frac{\partial}{\partial
\xi_1}]=\xi_3\frac{\partial}{\partial
\xi_2}-\xi_2\frac{\partial}{\partial \xi_3}$. It is involutive   but \emph{not} a distribution (compare also to Ex. \ref{gD}). The set of invariant functions $C(\cM)^{\cD}$ is $\{\alpha+\beta \xi_2\xi_3: \alpha,\beta\in \R \}$, so it is isomorphic
to the functions on  $\R[2]$ (with vanishing homological vector field).
\end{ep}

 Given a Lie algebroid $A$,
Jotz-Ortiz define the notion of IM-foliations \cite[Def. 5.1]{JotzOrtizFol}.

\begin{defi} \label{IM}
An \emph{IM-foliation} for the Lie algebroid $A \to M$ consists of
\begin{itemize}
 \item a Lie subalgebroid $B\to M$ of $A$,
\item an involutive distribution $F$ on $M$,
\item a flat $F$-connection $\nabla$ on the vector bundle  $A/B\to M$,
\end{itemize}
such that, denoting  $\Gamma_{\nabla}(A):=\{a\in \Gamma(A):   \nabla(a \text{ mod } B)=0 \}$: 
\begin{itemize}
 \item[(i)] if  $a,a' \in \Gamma_{\nabla}(A)$   then $[a,a']_A\in \Gamma_{\nabla}(A)$,
 \item[(ii)] if  $b\in \Gamma(B)$, and $a\in \Gamma_{\nabla}(A)$, then $[b,a]_A\in \Gamma(B)$,
 \item[(iii)] the anchor $\rho_A$ maps $B$ into $F$,
 \item[(iv)] if $a \in \Gamma_{\nabla}(A)$ and $Z\in \Gamma(F)$, then $[\rho_A(a),Z]\in \Gamma(F)$.
 \end{itemize}
\end{defi}

\begin{remark}
IM-foliations are the infinitesimal counterparts of ideal systems (Def. \ref{is}). More precisely: if $(B,F,\nabla)$ is an IM-foliation for which $F$ is simple and $\nabla$ has no holonomy, then $(B,R,\Theta)$ is an ideal system,  where $R:=\{(x,x'):\pi(x)=\pi(x')\}$ (for  $\pi \colon M\to M/F$ the projection) and  $\Theta$ is the  action of $R$ on   $A/B\to M$ by parallel transport along $\nabla$.  
\end{remark}

The relevance of  IM-foliations is that they are in bijective correspondence with morphic foliations on $A$ and - when $A$ is integrable - with multiplicative foliations on the source simply connected Lie groupoid integrating $A$ \cite{JotzOrtizFol}. We now show that they are also in correspondence with involutive distributions on $A[1]$ which are preserved by $Q$.

\begin{prop}
Given a Lie algebroid  $A\to M$, there is a bijection between 
\begin{itemize}
\item   involutive distributions $\cD$ on $A[1]$ such that $[Q,\cD]\subset \cD$,
\item IM-foliations on $A$.
\end{itemize} 
\end{prop}
\begin{proof}
Given an   involutive distribution on $A[1]$, consider the triple
 $(F,B,\nabla)$ encoding it by Lemmas \ref{distrCDO} and \ref{BF}.
That this triple is an IM-foliation was checked in  the proof of Prop. \ref{distri} (notice that the arguments in the proof do not use that $F$ is simple and $\nabla$ has no holonomy).

For the converse, let $(F,B,\nabla)$ be an IM-foliation, and denote by $\cD$ the involutive distribution on $A[1]$ given by Lemmas \ref{distrCDO} and \ref{BF}. We will use repeatedly that $\cD_0$ is given by the subset of $\chi_0(\cM)= CDO(A)$ specified in \eqref{Cas}.
Our aim is to show that $[Q,\cD]\subset \cD$.
It is sufficient to show this inclusion for $\cD_{-1}$ and $\cD_0$, as they generate $\cD$ as a $C(\cM)$-module. 

First we show $[Q,\cD_{-1}]\subset \cD_0$.
Let $b\in \Gamma(B)= \cD_{-1}$. Then using ii) and iii) in Def. \ref{IM} one sees that  $[Q,b]\in \cD_0$. 

The space $\cD_1$ of degree 1 elements of $\cD$ can be described as 
\begin{equation}\label{PP}
\cD_1=\{P\in \chi_1(\cM): [P,a]\in \cD_0 \text{ for all }a\in \Gamma_{\nabla}(A)\}.
\end{equation}
Using this we show that $[Q,\cD_{0}]\subset \cD_1$: let $X_0\in \cD_0$ and  $a\in \Gamma_{\nabla}(A)$. We have $$[[Q,X_0],a]=[Q,[X_0,a]]+[[Q,a],X_0].$$ The first term on the r.h.s. lies in $\cD_0$ since $[X_0,a]\in \cD_{-1}$. The second term on the r.h.s. is $[ad_a,X_0]$, and is seen to lie in $\cD_0$ using i),ii) and iv) of Def. \ref{IM}.

Last, we prove eq. \eqref{PP}:  the inclusion ``$\subset$'' follows immediately from $\cD_1=C_1(\cM) \cD_0$. For the opposite inclusion, 
fix locally a frame $a_i$ of $A$ consisting of elements of $\Gamma_{\nabla}(A)$, and denote by $\xi_i\in \Gamma(A^*)=C_1(\cM)$ the dual frame.  Any $P\in \chi_{1}(\cM)$ can be written as
\begin{equation*}
P=\sum_i X^i\xi_i+\frac{1}{2}\sum_{i,k} b^{ik}\xi_i\xi_k
\end{equation*}
where $X^i:= [P,a_i]\in \chi_{0}(\cM)$ and $b^{ik}:=[[P,a_i],a_k]\in \chi_{-1}(\cM)$. This ``Taylor expansion'' identity is proven noticing   that any  $P\in \chi_{1}(\cM)$ is determined by the values of $[P,a_i]$ for all $i$ (this is clear in coordinates), and checking by direct computation that these values coincide for both sides of the identity. Now, if $P$  belongs to the r.h.s. of eq. \eqref{PP}, 
then  $X^i\in \cD_0$ and $b^{ik}\in \Gamma(B)=\cD_{-1}$, showing that $P\in C_1(\cM) \cD_0=\cD_1$.
\end{proof}

\section{Actions}\label{sec:ac} 

In this section we consider distributions that arise from certain kinds of (infinitesimal) actions. In \S\ref{defact} we define (non-strict) actions of strict Lie-2 algebras on NQ-1 manifolds. In \S\ref{quotmu} we notice that such actions do not define involutive distributions in general. We give a sufficient condition for this to happen, and use Prop. \ref{exa} to perform reduction. Some examples are presented in \S\ref{exa}.

\subsection{Actions on   NQ-1 manifolds}\label{defact}

Recall that an  \emph{$L_{\infty}$-algebra}\footnote{More precisely,
  this is a  so-called \emph{flat}
  $L_\infty$-algebra, and is also the original definition. Flat means that the $0$th-bracket (or curvature) vanishes.
   All
  $L_\infty$-algebras appearing in this paper (for example  DGLAs of vector fields on NQ-manifolds) are of this kind.}
  is a  graded vector space $L=\bigoplus_{i\in \Z}L_i$ endowed with a sequence of multi-brackets ($n\ge 1$)
\begin{equation*}
[\dots ]_n \colon \wedge ^n L \to L
\end{equation*}
of degree $2-n$, satisfying the quadratic relations specified in \cite[Def. 2.1]{LadaMarkl}\footnote{\cite{LadaMarkl} uses a grading opposite to ours.}.
 Here $\wedge ^n L$ denotes the $n$-th graded skew-symmetric product of $L$.
When $[\dots ]_n=0$ for $n\ge3$ we recover the notion of DGLA (Def. \ref{dgla}), which is the one of interest in this note. An \emph{$L_{\infty}$-morphism} $\phi \colon L \rightsquigarrow L'$  between $L_{\infty}$-algebras is a sequence of maps ($n\ge 1$)
\begin{equation*}
\phi_n \colon \wedge ^n L \to L'
\end{equation*}
of degree $1-n$, satisfying certain relations
(see \cite[Def. 5.2]{LadaMarkl} in the case when $L'$ is a DGLA).

 \begin{defi}\label{def:DGLAaction}
Let $L$ be a $L_\infty$-algebra   and $\cM$ be an   NQ-manifold.
An \emph{action} of $L$ on $\cM$ is an $L_{\infty}$-morphism 
$$\phi \colon L \rightsquigarrow (\chi(\cM), d_Q:=[Q,-], [-,-])$$
where the right hand side is the DGLA of vector fields on $\cM$ as in Lemma \ref{lemma:vf}.
\end{defi}

\begin{remark}
 Even for $L$ a DGLA, such action $\phi$ might not be a strict
DGLA-morphism. There  are important instances of this. For
example,  given a (ordinary) Lie algebra $\g$ and a Poisson manifold
$(M,\pi)$, \v{S}evera \cite{Se}
 defines an \emph{up to homotopy Poisson action} as an $L_{\infty}$-morphism
 $\g \rightsquigarrow (C(T^*[1]M)[1],[\cdot,\cdot]_{S},[\pi,\cdot]_{S})$,
 where  $[\cdot,\cdot]_{S}$ denotes the Schouten bracket.
This is equivalent to the special case of   Def. \ref{def:DGLAaction} in which $(\cM,Q)=(T^*[1]M, [\pi,\cdot]_{S})$ and the action is Hamiltonian.
\end{remark}
 
\begin{defi}\label{strict}

A \emph{strict Lie 2-algebra}\footnote{The term  \emph{Lie 2-algebra}
 denotes an $L_{\infty}$-algebra concentrated in degrees $-1$ and
  $0$  {(see \cite{baez:2algebras})}. So a strict Lie 2-algebra is a Lie 2-algebra for which $[\dots ]_3=0$.}
   is a DGLA (see Def. \ref{dgla}) concentrated in degrees $-1$ and $0$.\end{defi}

From now on we will restrict ourselves to the case in which 
  $\cM$ is an NQ-1 manifold and $(L=L_{-1}\oplus L_{0}, \delta:=[\cdot]_1, [\cdot,\cdot]:=[\cdot,\cdot]_2)$  is a
 {strict} Lie 2-algebra.  One motivation for having  $L$ concentrated in degrees $-1$ and $0$  is that if a DGLA acts strictly and  almost freely  on $\cM$ (see Def. \ref {def:strict}) then it must be subject to this degree constraint.
  
An action
\begin{equation}\label{nq1action}
\phi \colon L_{-1}\oplus L_{0} \rightsquigarrow \chi(\cM).
\end{equation}
is an $L_{\infty}$-morphisms
between DGLAs.
We spell out its components. 
By degree reasons, the only non-zero components of $\phi$ are $\phi_1=:\mu$ and $\phi_2=:\eta$. More explicitly, 
\begin{align*}
  \mu\colon & L_{0} \rightarrow  \chi_{0}(\cM) \;\;\;\;\\
    \mu\colon & L_{-1}  \rightarrow  \chi_{-1}(\cM) \;\;\;\; \\
 \eta\colon&  \wedge^2 L_0 \rightarrow \chi_{-1}(\cM),
\end{align*}
subject to the constraints
\begin{align}
\label{constr1}
 d_Q \mu=& \mu \delta \\
\label{constr2}
 \mu [x,y]- [\mu x,\mu y]=&d_Q (\eta
(x \wedge y))\;\;\;\;\;\;\; \forall x,y\in L_{0},\\
 \label{constr3}
\mu [w,x]- [\mu w,\mu x]=& \eta (\delta w \wedge
x)\;\;\;\;\;\;\;\;\;\;\;  \forall w\in L_{-1}, x \in L_{0}, 
\end{align}
as well as
an equation for $x\wedge y \wedge z \in \wedge^3L_{0}$:
\begin{align} \label{constr4}
0=&\eta(x\wedge[y,z])
-\eta(y\wedge[x,z])+\eta(z\wedge[x,y])  \\
+&[\mu(x),\eta(y\wedge z)]-[\mu(y),\eta(x\wedge z)]
+[\mu(z),\eta(x\wedge y)]  \nonumber.
\end{align}
Notice that condition \eqref{constr1} says that $\mu$ is a map of complexes, 
\eqref{constr2} says that $\mu|_{L_{0}}$ is a morphism of Lie
algebras up to homotopy, and \eqref{constr3} says that $\mu|_{L_{-1}}$ is a morphism of Lie
modules up to homotopy.

\begin{remark}
  By eq. \eqref{constr1}, the image of the action map $\mu$ will be contained in the truncated DGLA
$\chi_{-1}(\cM) \oplus  \{X \in \chi_{0}(\cM): [Q,X]=0 \}.$   Hence  
actions of strict Lie 2-algebra on $\cM$ can be formulated using only the truncated DGLA.
\end{remark}


\subsection{Quotients by   actions}\label{quotmu}

We   define the distribution associated to an  action of a strict Lie 2-algebra, and study the corresponding quotient.

All along this subsection we consider  a strict Lie 2-algebra $L$ (Def. \ref{strict}) and an NQ-1 manifold  $\cM$, equal to $A[1]$ for some Lie algebroid $(A,[\cdot,\cdot]_A,\rho_A)$.
Let
 $$\phi:=(\mu, \eta) \colon L \rightsquigarrow \chi(\cM)$$ be an action as in \eqref{nq1action}.
Notice that the $C(\cM)$-submodule of $\chi(M)$ generated by  $\mu(L)$ and
$\eta(\wedge^2 L_0 )$ is usually not a distribution on $\cM$ (this
problem arises already 
in the case of  Lie algebra actions on  ordinary manifolds).  
Moreover it has two defects: first it is not involutive  (as suggested by eq. \eqref{constr2}), and second the operator
$d_Q:=[Q,\cdot]$ does not preserve its sections. A counterexample for both
defects is given in Ex. \ref{notinv} below. {This hints to a possible up-to-homotopy version of the concept of
  involutivity, which we delay to  later investigation.}
{For the moment, we might consider a ``completion''}  of the above
$C(\cM)$-module:
\begin{equation}\label{defD}
\cD:=Span_{C(\cM)}\{\mu(L)\;,\; \eta(\wedge^2 L_0)\;,\; d_Q (\eta( \wedge^2 L_0))\}.
\end{equation}
\emptycomment{\mcomment{About we could use $\eta( L_0\wedge \delta L_{-1})$ too}}

Unfortunately, the   fact that $\phi$ is an $L_{\infty}$-morphism (the constraints \eqref{constr1}-\eqref{constr4}) does not imply that $\cD$ is involutive. A counterexample is given by Ex. \ref{notin} below. Making an additional assumption, we can achieve that $\cD$ is involutive and $Q$-invariant:
\begin{prop}\label{goodaction}
If $\phi$ is a   action of the strict  Lie 2-algebra $L$ on $\cM$
for which 
$[\cD_0,\cD_{-1}]\subset \cD_{-1}$,  then\\
 i) $[Q,\cD]\subset \cD$\\
  ii) $\cD$ is involutive
\end{prop}
\begin{proof}
i) follows from the facts that  $\mu$ satisfies eq.
\eqref{constr1} and that $d_Q^2=0$.

 ii) Since  $[\cD_0,\cD_{-1}]\subset \cD_{-1}$ by assumption, and     $[\cD_{-1},\cD_{-1}]\subset \cD_{-2}=\{0\}$ by degree reasons, we just need to check
 that $\cD_0$ is closed under the bracket, i.e., we just need to consider $\mu(L_0)$
 and $ d_Q (\eta( \wedge^2 L_0))$.
Let $x,y\in L_0$ and $m, n \in L_0 \wedge L_0$. We have $[\mu(x),\mu(y)]\in \cD_0$ by eq. \eqref{constr2}.
 Further
$$[\mu(x),[Q,\eta( m )]]= [Q,[\mu(x),\eta(
m )]]- [[Q,\mu(x)],\eta( m )]$$ also lies in $\cD_0$. Indeed
  $[Q,\mu(x)]=0$ by eq. \eqref{constr1}, and the first term on the r.h.s. lies
in $\cD_0$ because of  $[\cD_0,\cD_{-1}]\subset \cD_{-1}$
 and because of i).

Last, $[[Q,
\eta(m)],[Q, \eta(n)]]= [Q,[[Q,
\eta(m)], \eta(n)]]\in \cD_0$, again  because of  $[\cD_0,\cD_{-1}]\subset \cD_{-1}$
 and   i).
\end{proof}

We summarize the conditions under which we can nicely
quotient $\cM$ by the  action $\phi$:
\begin{cor}\label{corquot}
Let $\phi$ be an action of the strict Lie 2-algebra $L$ on $\cM=A[1]$
such that   $\cD$ is a distribution, $[\cD_0,\cD_{-1}]\subset \cD_{-1}$ and the
 assumptions  of Prop. \ref{distriE}   are satisfied.
Then    $\cM/\cD$ is an NQ-1 manifold. 
It corresponds to the quotient of $A$ by the ideal system given by
\begin{itemize}
\item
the Lie subalgebroid $B=span\{\mu(L_{-1}), \eta(\wedge^2 L_0)\}$ of $A$
\item the Lie subgroupoid $R$ of $M\times M$ associated to the integrable distribution\\
 $F:=span\{\un{\mu(L_0)},\rho_A(\eta(\wedge^2 L_0))\}$ on $M$
\item the  action of $R$  on $A/B$ induced by
\noindent $span_{C^{\infty}(M)}\{\mu(L_0),d_Q(\eta(\wedge^2 L_0))\}\subset \chi_0(\cM)\cong CDO(A)$ (where the identification is given in by Lemma \ref{N1}).
\end{itemize}
\end{cor}
\begin{proof}
$\cM/\cD$ is an NQ-1 manifold by  Prop. \ref{quot}, whose assumptions are satisfied because of Prop. \ref{goodaction}.
  By  the same Prop. \ref{quot}   $\cM/\cD$ corresponds to the quotient of $A$ by
the ideal system associated to $\cD$
as in Prop. \ref{distri}, which is the above ideal system.
\end{proof}

We specialize further the action:

\begin{defi}\label{def:strict}
 A \emph{strict action} is  a morphism of DGLAs $ \mu \colon L \to \chi(\cM)$, or equivalently an action (in the sense of Def. \ref{def:DGLAaction}) for which $\eta=0$.

A strict action is \emph{almost free} if the map  $L \to T_m\cM=T_mM\oplus A_m[1], v\mapsto (\mu(v))|_m$ is injective 
for all $m\in M$.
 \end{defi}

When the action is strict,  we showed in \cite[\S 2.3.2]{ZZL1} that there is an induced action
$\Psi  \colon (L_{-1} \rtimes  G) \times A \rightarrow A$, which furthermore is an
$\mathcal{LA}$-group action\footnote{That is, an action in the category of Lie algebroids}. The latter is called a morphic action   \cite[Def. 3.0.14]{lucathesis}. From
 \cite[Thm. 3.2.1]{lucathesis} in Stefanini's thesis,
it follows that, when the action $\Psi$ is free and proper,  $A/(L_{-1} \rtimes  G) \to M/G$ is again a Lie algebroid with the property\footnote{Even more, it is a Lie algebroid fibration  \cite[Def. 1.1]{brahic-zhu}, which means that the projection  $A/(L_{-1} \rtimes  G) \to M/G$ admit a complete Ehresmann connection. Such a map is called a ``fibration'' in the sense that it introduces the expected long exact sequence of homotopy groups of Lie algebroids.}
that the projection of $A$ onto the quotient is a Lie algebroid morphism.

\begin{cor}\label{usLuca} Let $L$ be a strict Lie-2 algebra, $\cM:=A[1]$ a NQ-1 manifold, and
let $\mu \colon L\to \chi(\cM)$ be  a strict   action.
Assume that the action is   locally free.
Further assume that the induced Lie group action
$\psi \colon G \times A \to A$  obtained restricting $\Psi$ is free and proper.

Then the Lie algebroid corresponding to the NQ-1 manifold $\cM/\cD$  agrees with Stefanini's quotient of $A$ by the  $\mathcal{LA}$-group action $\Psi$.
\end{cor}
\begin{proof}
 We have
 $\cD=Span_{C(\cM)}\{\mu(L)\}$.
By the local freeness assumption, the image under $\mu$ of a basis of $L$ provides a set of local homogeneous generators of $\cD $ whose evaluations at points of $M$ are linearly independent, hence $\cD$ is a distribution. $\cD$ is involutive since $\mu$ preserves brackets. The leaves of the distribution $F$ on $M$ are just the orbits of the free action $\psi|_M$ (the restriction of the $G$-action $\psi$ to $M$), so  $M/F$ is a smooth manifold. Let $B=span\{\mu(L_{-1})\}$,
then the holonomy of the partial connection $\nabla$ is
given by the action of $G$ on $A/B$ induced by $\psi$, so by the freeness of $\psi$ we conclude that the holonomy of $\nabla$ is trivial.

Hence we can apply Cor. \ref{corquot}.
The vector bundle obtained  quotienting $A$ by the  ideal system of Cor. \ref{corquot} agrees with the quotient of $A$ by the action $\Psi$, and the induced Lie algebroid structures agree because in both cases the projection map from $A$  is a Lie algebroid morphism.
\end{proof}

\subsection{Examples}\label{exa}

We present two examples of actions (as in Def. \ref{def:DGLAaction}) of a  strict Lie 2-algebra $L$ on an NQ-1 manifold  $\cM$.

 The first is an example where the image of the action   is a   distribution, which however is neither involutive nor
preserved by $[Q,\cdot]$.

\begin{ep}\label{notinv}[The image of the action is neither
involutive nor $Q$-invariant] Let $L=\R^2$ be the abelian 2-dimensional Lie algebra
(concentrated in degree zero), fix a basis $e_0$, $e_0'$. Let
$A=T\R^3$, so $\cM=T[1]\R^3$,
 on which we take the standard degree $0$ coordinates $x_1,x_2,x_3$
and the corresponding degree $1$ coordinates $\xi_i(=dx_i)$. The deRham vector field on $\cM$ reads $Q=\sum_i
\xi_i \frac{\partial}{\partial x_i}$.
Define
$\phi:=(\mu, \eta) \colon L \rightsquigarrow \chi(\cM)$ by\footnote{Notice that, by Cartan's formula, $\mu(e_0)$ acts on $C(\cM)=\Omega(\R^3)$  by the Lie derivative $\cL_{ \frac{\partial}{\partial  x_1}}$.}
\begin{align*}
  \mu(e_0)=&[Q,\frac{\partial}{\partial  \xi_1}]=\frac{\partial}{\partial  x_1}\\
\mu(e_0')=& [Q,x_1 \frac{\partial}{\partial \xi_3}-\frac{\partial}{\partial \xi_2}]=x_1 \frac{\partial}{\partial x_3}-\frac{\partial}{\partial x_2}+\xi_1 \frac{\partial}{\partial \xi_3},
\\
 \eta(e_0\wedge e_0')=&-\frac{\partial}{\partial \xi_3}.
\end{align*} 
One checks that $\phi$ is an $L_{\infty}$-morphism (conditions \eqref{constr1}-\eqref{constr4}).
 The $C(\cM)$-span of the image of
$\phi$ is $span_{C(\cM)}\{\frac{\partial}{\partial x_1},
  x_1 \frac{\partial}{\partial x_3}-\frac{\partial}{\partial x_2},
  \frac{\partial}{\partial  \xi_3}\}$.
It is a distributon, but is not involutive (because its restriction to the body $\R^3$ is   not involutive) nor
preserved by $[Q,\cdot]$.

On the other hand,
the distribution $\cD$ (defined in eq. \eqref{defD})  on $\cM$ is spanned by $
\frac{\partial}{\partial x_1},\frac{\partial}{\partial
x_2},\frac{\partial}{\partial x_3},\frac{\partial}{\partial \xi_3}$.
 $\cD$ is   involutive and its sections are preserved by $[Q,\cdot]$. The quotient $\cM/\cD$ is isomorphic to $\R^2[1]$ with vanishing
 homological vector field, which corresponds to the abelian 2-dimensional
 Lie algebra.
\end{ep}

Second, we display an example where the $C(\cM)$-module $\cD$ is   not involutive.

\begin{ep}\label{notin}[$\cD$ is not involutive]
Consider the strict Lie-2 algebra $L=\R[1]\oplus \R^2$, with zero bracket and differential
$\delta \colon \R[1] \to  \R^2,e_{-1}\mapsto e_0$, where $\{e_{-1}\}$ and $\{e_0,e'_0\}$ are bases of $\R$ and $\R^2$ respectively. Let $M$ be a manifold and $X,Y$ vector fields such that $[X,[X,Y]_M]_M$ does not lie in the
the $C^{\infty}(M)$-span of $X$ and $[X,Y]_M$. Take $\cM=(T[1]M,Q)$ where $Q$ is the deRham vector field. By  $[\cdot,\cdot]_M$ we denote the Lie bracket of vector fields on $M$, while by $[\cdot,\cdot]$ we denote the graded Lie   bracket of vector fields on $\cM$.

Consider the   action
\begin{align*}
\phi:=(\mu, \eta)\colon \R[1]\oplus \R^2 \rightsquigarrow \chi(\cM)
\end{align*}
given by
\begin{align*}
  \mu(e_{-1})&=\bar{X}\\
  \mu(e_0)&=[Q,\bar{X}]\\
\mu(e_0')&= [Q,\bar{Y}]\\
 \eta(e_0\wedge e_0')&=- \overline{[X,Y]_M}.
\end{align*}
Here the overline denotes the identification $\Gamma(TM)\cong\chi_{-1}(\cM), X \mapsto \bar{X}$ as in Prop. \ref{N1}.
It is easy to check that $\mu$ is an $L_{\infty}$-morphism,
i.e. that  eq. \eqref{constr1}-\eqref{constr4} are satisfied.
However $[\cD_0,\cD_{-1}]\not\subset \cD_{-1}$, as $$[\mu(e_0),\eta(e_0\wedge e_0')]=-[[Q,\bar{X}],\overline{[X,Y]_M}]=-\overline{[X,[X,Y]_M]_M}$$ does not lie in $\cD_{-1}
=span_{C^{\infty}(M)}\{\bar{X},\overline{[X,Y]_M}\}.$ This shows that $\cD$ is not involutive.

More concretely, we can take $M=\R^3$ with coordinates $x_1,x_2,x_3$ and $X=\pd{x_1}$, $Y=\frac{x_1^2}{2}\pd{x_3}-x_1\pd{x_2}$, as they satisfy $[X,Y]_M=  {x_1}\pd{x_3}- \pd{x_2}$ and $[X,[X,Y]_M]_M=\pd{x_3}$.
\end{ep}

\bibliographystyle{../bib/habbrv}
\bibliography{../bib/bibz}

\end{document}